\documentclass[11pt]{article}
\usepackage{amsfonts, amssymb, graphicx, amsmath}
\usepackage[numbers, square]{natbib}
\usepackage[usenames]{color}
\usepackage[pdftex]{hyperref}
\hypersetup{
    colorlinks=true,       % false: boxed links; true: colored links
    linkcolor=blue,        % color of internal links
    citecolor=red,        % color of links to bibliography
    filecolor=blue,        % color of file links
    urlcolor=blue          % color of external links
}

\newtheorem{theorem}{Theorem}
\newtheorem{corollary}{Corollary}
\newtheorem{proposition}{Proposition}
\newtheorem{lemma}{Lemma}
\newtheorem{example}{Example}
\newtheorem{remark}{Remark}
\newtheorem{definition}{Definition}
\newcommand{\beq}{\begin{equation}}
\newcommand{\eeq}{\end{equation}}
\newcommand{\beas}{\begin{eqnarray*}}
\newcommand{\eeas}{\end{eqnarray*}}
\newcommand{\bea}{\begin{eqnarray}}
\newcommand{\eea}{\end{eqnarray}}
\newcommand{\bei}{\begin{itemize}}
\newcommand{\eei}{\end{itemize}}
\newcommand{\ben}{\begin{enumerate}}
\newcommand{\een}{\end{enumerate}}
\newcommand{\bet}{\begin{theorem}}
\newcommand{\eet}{\end{theorem}}
\newcommand{\bel}{\begin{lemma}}
\newcommand{\eel}{\end{lemma}}
\newcommand{\bep}{\begin{proposition}}
\newcommand{\eep}{\end{proposition}}
\newcommand{\bed}{\begin{definition}}
\newcommand{\eed}{\end{definition}}
\newcommand{\bec}{\begin{corollary}}
\newcommand{\eec}{\end{corollary}}
\newcommand{\bex}{\begin{example}}
\newcommand{\eex}{\end{example}}
\newenvironment{proof}[1][Proof]{\textbf{#1.} }{\ \rule{0.5em}{0.5em}}
\newcommand{\qed}{\quad\hbox{\vrule width 4pt height 6pt depth 1.5pt}}

\newcommand{\goto}{\rightarrow}

\newcommand{\hf}{{1\over 2}}
\def\sf{{\cal F}}

\def\limsup{\mathop{\overline{\rm lim}}}
\def\liminf{\mathop{\underline{\rm lim}}}

\begin{document}

\title{Law of Log Determinant of Sample Covariance Matrix and Optimal Estimation of Differential Entropy for High-Dimensional Gaussian Distributions}
% for High-Dimensional Gaussian Distributions}

%\title{Optimal Estimation of Log Determinant of Covariance Matrix}
\author{T. Tony Cai$^{1}$, Tengyuan Liang$^{1}$, and Harrison H. Zhou$^{2}$}
\date{}
\maketitle

\begin{abstract}

Differential entropy and log determinant of the covariance matrix of a multivariate Gaussian distribution have  many applications in coding, communications, signal processing and statistical inference. In this paper we consider in the high dimensional setting optimal estimation of the differential entropy and the log-determinant of the covariance matrix. We first establish a central limit theorem for the log determinant of the sample covariance matrix in the high dimensional setting where the dimension $p(n)$ can grow with the sample size $n$. An estimator of the differential entropy and the log determinant is then considered. Optimal rate of convergence is obtained. It is shown that in the case $p(n)/n \rightarrow 0$ the estimator is asymptotically sharp minimax. The ultra-high dimensional setting where $p(n) > n$ is also discussed.

\end{abstract}

\footnotetext[1]{
Department of Statistics, The Wharton School, University of Pennsylvania, Philadelphia, PA 19104. \newline
\indent \ \ 
The research of Tony Cai was supported in part by NSF Grant DMS-0604954 and NSF FRG Grant \newline
\indent \ \ DMS-0854973.} 
\footnotetext[2]{
Department of Statistics, Yale University, New Haven, CT 06511. The research of Harrison Zhou was \newline
\indent \ \ 
supported in part by NSF Career Award DMS-0645676 and NSF FRG Grant DMS-0854975.}



\newpage

%%%%%%%%%%%%%%%%%%%%%%%%%%%
\section{Introduction}
\label{introduction.sec}
%%%%%%%%%%%%%%%%%%%%%%%%%%%

The determinant of a random matrix is an important functional that has been actively studied  in random matrix theory under different settings. See, for example, \cite{goodman1963distribution, komlos1967determinant,komlos1968determinant, girko1980central,girko1990theory,girko1998refinement, delannay2000distribution, tao2006random,tao2012central, rouault2006asymptotic, nguyen2011random}.
In particular, central limit theorems for the log-determinant have been established for random Gaussian matrices  in \cite{goodman1963distribution}, for general real i.i.d. random matrices in \cite{nguyen2011random} under an exponential tail condition on the entries,  and for Wigner matrices in \cite{tao2012central}. 
The determinant of random matrices has many applications. For example, the determinant is needed for computing the volume of random parallelotopes, which is of significant interest in random geometry (see \cite{mathai1999random, nielsen1999distribution}). More specifically, let  $Z=(Z_1,\ldots,Z_p)$ be linearly independent random vectors in $\mathbb{R}^n$ with $p\le n$. Then the convex hull of these $p$ points in $\mathbb{R}^n$ almost surely determines a $p-$parallelotope and the volume of this random $p-$parallelotope is given by $\bigtriangledown_{n,p} = \det (Z^T Z)^{1/2}$, the squared root of the determinant of th random matrix $Z^T Z$. 

The  differential entropy and the determinant of the covariance matrix of a multivariate Gaussian distribution play a particularly important role in information theory and statistical inference. The differential entropy has a wide range of applications in many areas including coding, machine learning, signal processing, communications, biosciences and chemistry. See \cite{srivastava2008bayesian, gupta2010parametric,beirlant1997nonparametric,costa2004geodesic, misra2005estimation}. For example, in molecular biosciences, the evaluation of entropy of a molecular system is important for understanding its thermodynamic properties. In practice, measurements on macromolecules are often modeled as Gaussian vectors.  For a multivariate Gaussian distribution $\mathcal{N}_p(\mu, \Sigma)$, it is well-known that the differential entropy $\mathcal{H}(\cdot)$ is given by
\begin{equation}
\mathcal{H}(\Sigma) = \frac{p}{2}+\frac{p\log(2\pi)}{2}+\frac{\log \det \Sigma}{2}. \label{Entropy}
\end{equation}
In this case, estimation of the differential entropy of the system is thus equivalent to estimation of the log determinant of the covariance matrix from the sample. For other applications, the relative entropy (a.k.a. the Kullback-Leiber Divergence), which involves the difference of the log determinants of two covariance matrices in the Gaussian case, is important. The determinant of the covariance matrices  is also needed for constructing hypothesis tests in multivariate statistics (see  \cite{anderson2003introduction, muirhead1982aspects}). For example, the likelihood ratio test for testing linear hypotheses about regression coefficients in MANOVA is based on the ratio of the determinants of two sample covariance matrices  \cite{anderson2003introduction}.
%which satisfies the Wilk's lambda distribution under fixed $p$ Gaussian settings. 
In addition,  quadratic discriminant analysis, which is an important technique for classification, requires the knowledge of the difference of the log determinants of the covariance matrices of Gaussian distributions. 
For these applications, it is important to understand the properties of the log determinant of the sample covariance matrix. The high-dimensional setting where the dimension $p(n)$ grows with the sample size $n$ is of particular current interest.

Motivated by the applications mentioned above,  in the present paper we first study the limiting law of the log determinant of the sample covariance matrix for the high-dimensional Gaussian distributions.  Let $X_{1},\ldots ,X_{n+1}$ be an independent random sample from 	the $p$-dimensional Gaussian distribution  $\mathcal{N}_p(\mu, \Sigma)$.
%For convenience, we shall use $n=N-1$ throughout the paper. 
The sample covariance matrix is 
\beq
\label{sample.cov}
\hat{\Sigma} = \frac{1}{n}\sum_{k=1}^{n+1} (X_{k} - \bar{X})(X_{k} - \bar{X})^T.
\eeq
A central limit theorem is established for the log determinant of $\hat{\Sigma}$ in the high-dimensional setting where the dimension $p$ grows with the sample size $n$ with the only restriction that $p(n) \le n$. In the case when $\lim_{n\rightarrow \infty} \frac{p(n)}{n} = r$ for some $0 \leq r < 1$, the central limit theorem shows
\beq
\label{CLT}
\frac{\log \det \hat \Sigma  -\sum_{k=1}^p \log\left(1- \frac{k}{n}\right)- \log \det \Sigma}{\sqrt{-2 \log\left(1-\frac{p}{n}\right)}} \stackrel{L}{\longrightarrow} \mathcal{N}(0,1)  \quad\mbox{as $n\goto\infty$.}
\eeq
The result for the boundary case $p=n$ yields
\begin{equation}
\frac{\log \det \hat \Sigma  - \log (n-1)! + n\log n - \log \det \Sigma}{\sqrt{2\log n}} \stackrel{L}{\longrightarrow} \mathcal{N}(0,1), \quad\mbox{as $n\goto\infty$.}
\end{equation}
In particular, this result recovers the central limit theorem for the log determinant of a random matrix  with iid standard Gaussian entries. See  \cite{goodman1963distribution} and  \cite{nguyen2011random}. 

We then consider optimal estimation of the differential entropy and the log-determinant of the covariance matrix in the high dimensional setting. In the conventional fixed dimensional case,  estimation of the differential entropy  has been considered by using both Bayesian and frequentist methods. See, for example, \cite{misra2005estimation, srivastava2008bayesian, ahmed1989entropy}. A Bayesian estimator was proposed in \cite{srivastava2008bayesian} using the inverse Wishart prior which works without the restriction that dimension is smaller than the sample size. However, how to choose good parameter values for the inverse Wishart prior remains an open question when the population covariance matrix is nondiagonal. A uniformly minimum variance unbiased  estimator (UMVUE) was constructed in \cite{ahmed1989entropy}. It was later  proved in  \cite{misra2005estimation} that this UMVUE is in fact dominated by a Stein type estimator and is thus inadmissible. 
The construction of an admissible estimator was left as an open problem in \cite{misra2005estimation}. 

Based on the central limit theorem for the log determinant of  the sample covariance matrix $\hat{\Sigma}$, we consider an estimator of the differential entropy and the log determinant of $\Sigma$ and study its properties. A non-asymptotic upper bound for the mean squared error of the estimator is obtained. To show the optimality of the estimator, non-asymptotic minimax lower bounds are established using Cramer-Rao's Information Inequality. The lower bound results show that consistent estimation of $\log \det \Sigma$ is only possible when ${p(n) \over n} \goto 0$. Furthermore, it is shown that the estimator is asymptotically sharp minimax in the setting of ${p(n) \over n} \goto 0$.

The ultra-high dimensional setting where $p(n) > n$ is important due to many contemporary applications. It is a common practice in high dimensional statistical inference, including compressed sensing and covariance matrix estimation,  to impose structural assumption such as sparsity on the target in order to effectively estimate the quantity of interest. 
It is of significant interest to consider estimation of the log determinant of  the covariance matrix and the differential entropy  in the case $p(n)>n$ under such structural assumptions. A minimax lower bound is given in Section \ref{discussion.sec}  using Le Cam's method which shows that it is in fact not possible to estimate the log determinant consistently even when the covariance matrix is known to be diagonal with equal values. This negative result implies that consistent estimation of $\log \det \Sigma$ is not possible when $p(n) > n$ over all the collections of the commonly considered structured covariance matrices such as bandable, sparse, or Toeplitz covariance matrices.

The rest of the paper is organized as follows.  Section \ref{CLT.sec} establishes a central limit theorem for the log determinant of the sample covariance matrix. Section \ref{estimation.sec} considers optimal estimation of the differential entropy and the log-determinant of the covariance matrix. Optimal rate of convergence is established and the estimator is shown to be asymptotically sharp minimax when ${p(n)\over n} \goto 0$.
Section \ref{discussion.sec} discusses  related applications and  the case of $p(n)> n$. The proofs of the main results are given in Section \ref{proof.sec}.

%%%%%%%%%%%%%%%%%%%%%%%%%%%
\section{Limiting Law of the Log Determinant of the Sample Covariance Matrix}
\label{CLT.sec}
%%%%%%%%%%%%%%%%%%%%%%%%%%%

In this section, we consider the limiting distribution of the log determinant of the sample covariance matrix $\hat \Sigma$ and establish a central limit theorem for $\log\det \hat\Sigma$ in the high dimensional setting where $p(n)$ can grow with $n$ under the restriction  that $p(n) \leq n$. 

For two positive integers $n$ and $p$, define the constant $\tau_{n,p}$ by
\beq
\label{tau}
\tau_{n,p}:=\sum_{k=1}^p \left[ \psi\left(\frac{n-k+1}{2}\right)- \log \left(\frac{n}{2} \right)\right]
\eeq
where $\psi(x) = \frac{\partial} {\partial z} \log \Gamma(z) |_{z=x}$ is the Digamma function with $\Gamma(z)$ being the gamma function, and define the constant $\sigma_{n,p}$ by
\beq
\label{sigma}
\sigma_{n,p}:=\left(\sum_{k=1}^p \frac{2}{n-k+1}\right)^\hf.
\eeq
 
 We have the following central limit theorem for $\log\det \hat\Sigma$.
\begin{theorem}[Asymptotic Distribution]
\label{CLT.thm}
Let  $X_{1},\ldots ,X_{n+1} \stackrel{iid}{\sim} \mathcal{N}_p(\mu, \Sigma)$.  Suppose that $ n \rightarrow \infty$ and $p(n) \leq n$. Then the  log determinant of the sample covariance matrix $\hat \Sigma$ satisfies
\beq
\label{CLT2}
\frac{\log \det \hat \Sigma  -\tau_{n,p}- \log \det \Sigma}{\sigma_{n,p}} \stackrel{L}{\longrightarrow} \mathcal{N}(0,1) 
\quad\mbox{as $n\goto \infty$,}
\eeq
where the constants $\tau_{n,p}$ and $\sigma_{n,p}$ are given in \eqref{tau} and \eqref{sigma} respectively.
\end{theorem}

Note that Theorem \ref{CLT.thm} holds with either $p$ fixed or $p(n)$ growing with $n$, as long as $p(n) \le n$.
The assumption in Theorem \ref{CLT.thm} is generally mild. For example, it does not require that the limit of the ratio $\frac{p(n)}{n}$ exists. In particular, the theorem covers the following four special settings: (1) Fixed $p$;
(2) $\lim_{n\rightarrow \infty} \frac{p(n)}{n} = r$ for some $0 \leq r < 1$; (3) $p(n)<n$ and $\lim_{n\rightarrow \infty} \frac{p(n)}{n} = 1$; (4) The boundary case $p(n)=n$.

It is helpful to look at these special cases separately. Case (1) with fixed $p$ is the classical setting. In this case,  asymptotic normality of the determinant $\det\hat{\Sigma}$ has been well studied  \cite{anderson2003introduction,muirhead1982aspects}. For completeness, we state the result for $\log\det\hat \Sigma$ below.
\bec[Case (1): Fixed $p$]
If $p$ is fixed, then the log determinant of $\hat{\Sigma}$ satisfies
\beq
\frac{\log \det \hat \Sigma  - p(p+1)/(2n) - \log \det \Sigma}{\sqrt{2p/n }} \stackrel{L}{\longrightarrow} \mathcal{N}(0,1),  \quad\mbox{as $n\goto\infty$.}
\eeq
\eec

We now consider Case (2) where $\lim_{n\rightarrow \infty} \frac{p(n)}{n} = r$ for some $0 \leq r < 1$. 
It is easy to verify that in this case the constants $\tau_{n,p}$ and $ \sigma_{n,p} $  satisfy

\beq
\tau_{n,p} = \sum\limits_{k=1}^p \log\left(1- \frac{k}{n}\right)+O(\frac{1}{n}) \quad\mbox{and}\quad 
\sigma_{n,p} = \sqrt{-2 \log\left(1-\frac{p}{n}\right)}+O(\frac{1}{n}).
\eeq
It can be seen easily that $\tau_{n,p} \goto -\infty$ at the rate $O(n)$ when $0<r<1$.
We have the following corollary for Case (2), which reduces to \cite{jonsson1982some}.  
\bec[Case (2): $0 \leq r <1$]
\label{01.cor}
If $\lim_{n\rightarrow \infty} \frac{p(n)}{n} = r$ for some $0 \leq r < 1$, then the log determinant $\log \det \hat \Sigma$ satisfies
\beq
\frac{\log \det \hat \Sigma  - \sum_{k=1}^p \log\left(1- k/n\right)- \log \det \Sigma}{\sqrt{-2 \log\left(1-p/n\right)}} \stackrel{L}{\longrightarrow} \mathcal{N}(0,1)  \quad\mbox{as $n\goto\infty$.}
\eeq
\eec
Case (3) is more complicated. Unlike the other three cases, it cannot be reduced to a simpler form than the original Theorem \ref{CLT.thm} in general. We consider two interesting special settings: (a). ${p(n) \over n} \goto 1$ and $n-p(n)\rightarrow \infty$; (b). $n-p(n)$ is uniformly bounded. In case (a), the central limiting theorem is of the same form as in Corollary \ref{01.cor}. In case (b), the central limiting theorem is of the same form as the boundary case of $p(n)=n$ which is given as follows.
\bec[Boundary Case: $p(n) = n$]
\label{p=n.cor}
If $p(n)=n$, the log determinant $\log\det\hat{\Sigma}$ satisfies
\beq
\frac{\log \det \hat \Sigma  - \log (n-1)! + n\log n - \log \det \Sigma}{\sqrt{2\log n}} \stackrel{L}{\longrightarrow} \mathcal{N}(0,1), \quad\mbox{as $n\goto\infty$.}
\label{CLT.p=n}
\eeq
\eec

\medskip
It is interesting to note that the result given in \eqref{CLT.p=n} for the boundary case  $p(n)=n$ in fact recovers the central limit theorem for the log determinant of a random Gaussian matrix $Y = (y_{ij})_{n\times n}$ with iid $\mathcal{N}(0,1)$ entries $y_{ij}$,
\beq
\label{Gaussian.CLT}
\frac{\log |\det Y| - \frac{1}{2}\log (n-1)!}{\sqrt{\frac{1}{2} \log n}} \stackrel{L}{\longrightarrow} \mathcal{N}(0,1), \quad\mbox{as $n\goto\infty$.}
\eeq
See, for example,  \cite{goodman1963distribution} and \cite{nguyen2011random}. This can be seen as follows. When $p=n$,  it can be verified directly that the log determinant $\log \det \hat \Sigma$ satisfies
\begin{equation}
\log\det\hat{\Sigma}+ n \log n - \log\det \Sigma=\log\det (Y^TY)  = 2 \log |\det Y|
\end{equation}
where $Y$ is an $n \times n$ random matrix whose entries are independent standard Gaussian variables. Thus Corollary \ref{p=n.cor} yields
\beq
\frac{2 \log |\det Y| - \log (n-1)!}{\sqrt{2\log n}} \stackrel{L}{\longrightarrow} \mathcal{N}(0,1),
\eeq
which is equivalent to \eqref{Gaussian.CLT}.

%%%%%%%%%%%%%%%%%%%%%%%%%%%
\section{Estimation of Log-Determinant and Differential Entropy}
\label{estimation.sec}
%%%%%%%%%%%%%%%%%%%%%%%%%%%

As mentioned in the introduction, the log-determinant of the covariance matrix and differential entropy are important in many applications. In this section, we consider optimal estimation of the log-determinant and differential entropy of high-dimensional Gaussian distributions. Both minimax upper and lower bounds are given and the results yield sharp asymptotic  minimaxity.

Suppose we observe $X_{1},\ldots ,X_{n+1}\stackrel{iid}{\sim}  \mathcal{N}_p(\mu, \Sigma)$.
Based on the central limit theorem for $\log \det \hat \Sigma$ given in Theorem \ref{CLT.thm},
we consider the following estimator for the log determinant $T=\log \det\Sigma$ of the covariance matrix $\Sigma$, 
\begin{equation}
\hat{T} = \log \det \hat{\Sigma} - \tau_{n,p} \label{Estimator}
\end{equation}
and the corresponding estimator of the differential entropy $\mathcal{H}(\Sigma)$ given by
\begin{equation}
\widehat{\mathcal{H}(\Sigma)} = \frac{p}{2}+\frac{p\log(2\pi)}{2}+\frac{\log \det \hat \Sigma}{2}-{\tau_{n,p}\over 2}.
\label{Entropy.est}
\end{equation}
Here the constant $\tau_{n,p}$ as defined in \eqref{tau} can be viewed as a bias correction term. The estimators \eqref{Estimator} and \eqref{Entropy.est} have been studied in \cite{ahmed1989entropy, misra2005estimation} in the fixed dimensional setting. It was shown to be a UMVUE in  \cite{ahmed1989entropy} and inadmissible in \cite{misra2005estimation}.  
When the dimension $p$ is fixed, the bias correction term $\tau_{n,p}$ is of order $1\over n$ and is thus negligible. In particular, the log-determinant of the sample covariance matrix $ \log \det \hat{\Sigma}$ is asymptotically unbiased as an estimator of $ \log \det{\Sigma}$. Here we consider the estimator in the high dimensional setting where the dimension $p$ can grow with $n$ under the only  restriction $p(n) \le n$. The bias correction term $\tau_{n,p}$ plays a much more prominent role in such a setting because as discussed in Section \ref{CLT.sec},  $\tau_{n,p}$ is of order $n$ when $\lim_{n\rightarrow \infty} \frac{p(n)}{n} = r $ for some $0 < r < 1$.

In this section, we focus on the asymptotic behavior and optimality of the estimators $\hat{T}$ and $\widehat{\mathcal{H}(\Sigma)} $.  We establish a non-asymptotic upper bound for mean square error, a minimax lower bound and the optimal rate of convergence as well as sharp asymptotic  minimaxity for the estimators $\hat{T}$ and $\widehat{\mathcal{H}(\Sigma)} $ in the following two subsections. Since the  log determinant $\log \det\Sigma$ and  the differential entropy $\mathcal{H}(\Sigma)$ only differ by a constant in the Gaussian case, the two estimation problems are essentially the same. We shall focus on  estimation of $\log \det\Sigma$  in the rest of this section.

\subsection{Upper Bound}

We begin by giving  a non-asymptotic upper bound for the mean squared error of the estimator $\hat{T}$.
\begin{theorem}[Non-Asymptotic Upper Bound]
\label{NAsymUpperThm}
Suppose $p \leq n$. Let the estimator $\hat{T}$ be defined in \eqref{Estimator}. Then the risk of $\hat T$ satisfies
\begin{equation}
\label{upp.bnd}
\mathbb{E} \left( \hat{T} - \log \det \Sigma \right)^2 \leq -2 \log\left( 1- \frac{p}{n} \right) + \frac{10p}{3n} \cdot \frac{1}{n-p}.
\end{equation}
\end{theorem}

The proof of this theorem is connected to that of Theorem \ref{CLT.thm} as it can be seen intuitively that
\beq
\mathbb{E} \left( \hat{T} - \log \det \Sigma \right)^2 \sim \sigma^2_{n,p} = \sum_{k=1}^p \frac{2}{n-k+1} \leq -2 \log\left( 1- \frac{p}{n} \right)
\eeq
which yields the dominate term in \eqref{upp.bnd}. The higher order term on the right hand side of \eqref{upp.bnd} can be worked out explicitly using Taylor expansion with the remainder term. The detailed proof including derivation of the higher order term is given in Section \ref{proof.sec}.

\subsection{Asymptotic Optimality}

Theorem \ref{NAsymUpperThm} gives an upper bound for the risk of the estimator $\hat T$. We now establish  the optimal rate of convergence for estimating  $\log \det\Sigma$ by obtaining a minimax lower bound using the Information Inequality. The results show that the estimator $\hat T$ is asymptotically sharp minimax in the case $\lim\limits_{n\rightarrow \infty} \frac{p(n)}{n} = 0$. 

\begin{theorem}[Non-Asymptotic Information Bound]
\label{CRLowerThm}
Let  $X_{1},\ldots,X_{n+1} \stackrel{iid}{\sim} \mathcal{N}_p(\mu, \Sigma)$.  Suppose $p \leq n$. Then the minimax risk for estimating $\log \det\Sigma$ satisfies
\begin{equation}
\inf_{\delta} \sup_{\Sigma} \mathbb{E} (\delta- \log \det \Sigma)^2 \geq 2 \cdot \frac{p}{n}.
\end{equation}
where the infimum is taken over all measurable estimators $\delta$ and the supreme is taken over all the possible positive definite covariance matrix $\Sigma$.
\end{theorem}
The proof of Theorem \ref{CRLowerThm} is given in Section \ref{proof.sec}. A major tool is the Cramer-Rao Information Inequality.
Together with the upper bound given in \eqref{upp.bnd}, we have the following asymptotic optimality result.

\begin{theorem}[Asymptotic Optimality]
\label{AsympOptThm}
Let  $X_{1},\ldots ,X_{n+1} \stackrel{iid}{\sim} \mathcal{N}_p(\mu, \Sigma)$.  Suppose that $ n \rightarrow \infty$,  $p(n) \leq n$ and $n-p(n) \rightarrow \infty$. Then
\begin{equation}
\label{upper.lower.bnds}
2 \cdot \liminf_{n\rightarrow \infty} \frac{p}{n} \leq \varliminf_{n \rightarrow \infty} \inf_{\delta} \sup_{\Sigma} \mathbb{E} (\delta - \log \det \Sigma)^2 \leq \varlimsup_{n \rightarrow \infty} \inf_{\delta} \sup_{\Sigma} \mathbb{E} (\delta - \log \det \Sigma)^2 \leq  \limsup_{n\rightarrow \infty}\left(-2\log\left( 1- \frac{p}{n}\right)\right).
\end{equation}
In particular, if ${p(n)\over n} \goto 0$, then the minimax risk satisfies
\begin{equation}
\lim_{n \rightarrow \infty}  \frac{n}{p} \cdot \inf_{\delta} \sup_{\Sigma}  \mathbb{E} (\delta - \log \det \Sigma)^2 = 2 
\end{equation}
and the estimator $\hat T$ defined in \eqref{Estimator} is asymptotically sharp minimax.
\end{theorem}
Assume $\lim_{n\rightarrow \infty} \frac{p(n)}{n} = r \in [0,1)$. In the case of $r=0$, Theorem \ref{AsympOptThm} shows that the optimal constant in the asymptotic risk is $2$ and that the estimator $\hat T$ given in \eqref{Estimator} attains both the optimal rate and the optimal constant asymptotically.  It is thus asymptotically sharp minimax.  
%Theorem \ref{AsympOptThm} shows the strong  asymptotic optimality of the estimator  $\hat T$ given in \eqref{Estimator}. 
 When $0<r<1$,  the theorem also shows that the minimax risk is non-vanishing and is bounded between $2r$ and $-2\log(1-r)$. It is thus not possible to estimate $\log \det\Sigma$ consistently under the squared error loss in this case.

\begin{remark}{\rm
We have focused on the case $0 \le r < 1$ in Theorem \ref{AsympOptThm}.  
When $r=1$,  Theorem \ref{CRLowerThm} shows that
\begin{equation}
\varliminf_{n \rightarrow \infty}\inf_{\delta} \sup_{\Sigma} \mathbb{E} (\delta- \log \det \Sigma)^2 \geq 2.
\end{equation}
So consistent estimation of $\log \det \Sigma$ under mean squared error is not possible. If $r=1$ and $n-p$ is uniformly bounded, then
\[
2 \cdot \frac{p}{n} \leq  \inf_{\delta} \sup_{\Sigma} \mathbb{E} (\delta - \log \det \Sigma)^2 \le  c \log n
\]
for some positive constant $c$, which can be taken as $2$ as $n\goto \infty$.
}
\end{remark}

In terms of estimating the differential entropy  $\mathcal{H}(\Sigma)$, 
 the entropy estimator $\widehat{\mathcal{H}(\Sigma)}$ defined in \eqref{Entropy.est} is asymptotic optimal when ${p(n)\over n}\goto 0$, which means that in the asymptotic sense, $\widehat{\mathcal{H}(\Sigma)}$ is the optimal minimax estimator.

%%%%%%%%%%%%%%%%%%%%%%%%%%%
\section{Discussions}
\label{discussion.sec}
%%%%%%%%%%%%%%%%%%%%%%%%%%%

In this paper, we have focused on estimating the log determinant in the ``moderately" high dimensional setting under the restriction that $p(n)\le n$. The lower bound given in Theorem \ref{CRLowerThm} shows that it is not possible to estimate the log determinant consistently when ${p(n)\over n} \goto r > 0$. It is a common practice in high dimensional statistical inference to impose structural assumption such as sparsity on the parameter in order to effectively estimate the quantity of interest. In the context of covariance matrix estimation, commonly considered collections include  bandable covariance matrices, sparse covariance matrices, and Toeplitz covariance matrix. See, for example, \cite{cai2010optimal}, \cite{cai2012sparse}, and \cite{cai2012toeplitz}. It is interesting to see if the log determinant can be well estimated in the high dimensional case with $p(n) > n$ under one of these structural constraints. The answer is unfortunately negative. 

For any constant  $K>1$, define the following collection of $p$-dimensional bounded diagonal covariance matrices,
\beq
\mathcal{D}_K = \{{\rm diag}(\overbrace{a, a, \cdots, a}^{p}): 1/K \leq a \leq K\}.
\eeq
When $p(n)> n$, the following minimax lower bound shows that it is not possible to accurately estimate the log determinant even for the simple diagonal matrices in $\mathcal{D}_K$.
\begin{theorem}[Minimax Lower Bounds]
\label{MinimaxLowerThm}
Let  $X_{1},\ldots ,X_{n+1} \stackrel{iid}{\sim} \mathcal{N}_p(\mu, \Sigma)$. The minimax risk of estimating the log determinant of the covariance matrix $\Sigma$ over the collection  $\mathcal{D}_K$ of bounded diagonal matrices 
satisfies, 
\begin{equation}
\inf_{\delta}\sup_{\Sigma\in \mathcal{D}_K} \mathbb{E} \left( \delta - \log \det \Sigma \right)^2 \ge C_K \cdot \frac{p}{n},%\quad\mbox{as $n\rightarrow \infty$.}
\end{equation}
for all $n, p$, where $C_K$ is a constant satisfies $0 < C_K \leq 2$.
\end{theorem}
%Here the $\succsim$ notation means asymptotically greater than or  equal to,  i.e. $A(n) \succsim C \cdot B(n)$ meaning $\liminf_{n \rightarrow \infty} \frac{A(n)}{B(n)} \geq C$.
The proof of this minimax lower bound is given in Section \ref{proof.sec} using Le Cam's method. Theorem \ref{MinimaxLowerThm} shows that when $p(n) > n$ it is not possible to estimate consistently the bounded diagonal matrices in $\mathcal{D}_K$. Since all the reasonable collections of covariance matrices including the three collections mentioned earlier contain $\mathcal{D}_K$ as a subset, it is thus also impossible to estimate $\log \det \Sigma$ consistently over those commonly used collection of covariance matrices when the dimension is larger than the sample size.

In addition to the differential entropy considered in this paper, estimating the log determinant of  covariance matrices is needed for many other applications. One common problem in statistics and engineering is to estimate the distance between two population distributions based on the samples. A commonly used measure of closeness is the relative entropy or the Kullback-Leiber Divergence. For two distributions $\mathbb{P}$ and $\mathbb{Q}$ with respective density functions $p(\cdot)$ and $q(\cdot)$, the relative entropy between  $\mathbb{P}$ and $\mathbb{Q}$ is
\begin{equation}
KL(\mathbb{P},\mathbb{Q}) = \int p(x) \log \frac{p(x)}{q(x)} dx.
\end{equation}
In the case of two multivariate Gaussian distributions $\mathbb{P}=\mathcal{N}_p(\mu_1, \Sigma_1)$ and $\mathbb{Q}=\mathcal{N}_p(\mu_2, \Sigma_2)$, 
\begin{equation}
KL(\mathbb{P},\mathbb{Q}) = \frac{1}{2} \left( {\rm tr}\left( \Sigma_2^{-1}\Sigma_1\right) - p + (\mu_2 -\mu_1)^T \Sigma_2^{-1} (\mu_2 -\mu_1) + \log \left( \frac{\det \Sigma_1}{\det \Sigma_2}\right)\right).
\label{KL}
\end{equation}
From (\ref{KL}), it is clear that estimation of the relative entropy involves estimation of the  log determinants $\log\det\Sigma_1$ and $\log\det\Sigma_2$. The results given in this paper can be readily used for this part of the estimation problem.

The estimation results obtained in the present paper can also be applied for testing  the hypothesis that two multivariate Gaussian distributions $\mathbb{P}=\mathcal{N}_p(\mu_1, \Sigma_1)$ and $\mathbb{Q}=\mathcal{N}_p(\mu_2, \Sigma_2)$ have the same entropy,
\begin{equation}
H_0: \mathcal{H}(\mathbb{P}) = \mathcal{H}(\mathbb{Q}) \quad\mbox{vs.}\quad  H_1: \mathcal{H}(\mathbb{P}) \neq \mathcal{H}(\mathbb{Q}).
\end{equation}
For any given significance level $0<\alpha<1$, a test with the asymptotic level $\alpha$ can be easily constructed using the central limit theorem given in Section \ref{CLT.sec}, based on two independent samples, one from  $\mathbb{P}$ and another from  $\mathbb{Q}$.
%\nr{We could either stop here or continue with more details on the test.}
%Given two independent random samples, $X_{1},\ldots ,X_{N_1} \stackrel{iid}{\sim} \mathcal{N}_p(\mu_1, \Sigma_1)$ and $Y_{1},\ldots ,Y_{N_2} \stackrel{iid}{\sim} \mathcal{N}_p(\mu_2, \Sigma_2)$ ... 

Knowledge of the log determinant of covariance matrices is also essential for the quadratic discriminant analysis (QDA). For classification of two multivariate Gaussian distributions $\mathcal{N}_p(\mu_1, \Sigma_1)$ and $\mathcal{N}_p(\mu_2, \Sigma_2)$, when the parameters $\mu_1, \mu_2, \Sigma_1$ and $\Sigma_2$ are known, the oracle discriminant is
\begin{equation}
\label{QDA}
\Delta = -(z-\mu_1)^T\Sigma_1^{-1}(z-\mu_1)+(z-\mu_2)^T\Sigma_2^{-1}(z-\mu_2) - \log\left( \frac{\det\Sigma_1}{\det\Sigma_2} \right).
\end{equation}
That is, the observation $z$ is classified into the population with $\mathcal{N}_p(\mu_1, \Sigma_1)$ distribution if $\Delta > 0$ and into $\mathcal{N}_p(\mu_2, \Sigma_2)$ otherwise. In applications, the parameters are unknown and the oracle discriminant needs to be estimated from data. One of the importantr quantities in \eqref{QDA} involves the log determinants. Efficient estimation of $\log\det\Sigma_1 - \log\det\Sigma_2$ leads to a better QDA rule.

%%%%%%%%%%%%%%%%%%%%%%%%%%%
\section{Proofs}
\label{proof.sec}
%%%%%%%%%%%%%%%%%%%%%%%%%%%

We give the proofs of the main results in this section. We begin by collecting two basic but important  lemmas for the proof of Theorem  \ref{NAsymUpperThm}.

\begin{lemma}
\label{StandardLma}
Let $X_{1},\ldots,X_{n+1} \stackrel{iid}{\sim} \mathcal{N}_p(\mu, \Sigma)$ with $p\leq n$. Denote the sample covariance matrix by $\hat{\Sigma}$. Then
\begin{equation}
\log \det \hat{\Sigma} - \log \det \Sigma = \log \det \hat{I}
\end{equation}
where $\hat{I} = \frac{1}{n} \sum_{k=1}^n Y_k Y_k^T$ is the sample covariance matrix for independent and identically distributed $p$-variate Gaussian random variables $Y_{1},\ldots,Y_{n} \stackrel{iid}{\sim} \mathcal{N}_p(0, I)$, $I$ is the identity matrix.
\end{lemma}
\begin{proof}
Note that the distribution of $\hat{\Sigma} = \frac{1}{n}\sum_{k=1}^{n+1} (X_k-\bar{X})(X_k-\bar{X})^T$ is the same as $ \frac{1}{n} \sum_{k=1}^n Z_k Z_k^T$, where $Z_{1},\ldots,Z_{n} \stackrel{iid}{\sim} \mathcal{N}_p(0, \Sigma)$. See, for example, \cite{anderson2003introduction}.
Define $Y_k = \Sigma^{-1/2} Z_k$, then $Y_{1},\ldots,Y_{n} \stackrel{iid}{\sim} \mathcal{N}_p(0, I)$ and
\beas
\log \det \hat{\Sigma} - \log \det \Sigma & = & \log \det \left( \Sigma^{-1/2} \hat{\Sigma} \Sigma^{-1/2}  \right) \\ 
& = &  \log \det \left( \Sigma^{-1/2} \left(\frac{1}{n} \sum_{k=1}^n Z_k Z_k^T\right)
\Sigma^{-1/2} \right) \\
& = &  \log \det \left( \frac{1}{n} \sum_{k=1}^n Y_k Y_k^T \right) \\
& = & \log \det\hat{I}. 
\eeas
\end{proof}

\bigskip
A variant of the well-known Bartlett decomposition \cite{anderson2003introduction} in multivariate statistics implies the following lemma on the distribution of the determinant of the sample covariance matrix.
\begin{lemma}
\label{BartlettLma}
The law of $\log \det \left(n\hat{I}\right)$ is the same as the sums of $p$-independent $\log \chi^2$ distribution, namely
\begin{equation}
\log \det \left(n\hat{I}\right) \stackrel{L}{=} \sum_{k=1}^p \log \left( \chi_{n-k+1}^2 \right)
\end{equation}
where $\chi^2_{n},\ldots,\chi^2_{n-p+1}$ are mutually independent $\chi^2$ distribution with the degrees of freedom $n,\ldots, n-p+1$ respectively. 
\end{lemma}

\subsection{Proof of Theorem \ref{CLT.thm}}

The proof of Theorem \ref{CLT.thm} relies on the above two lemmas, the following Lemma \ref{real.lemma} and an analysis of the characteristic functions.
\begin{lemma}
\label{real.lemma}
\begin{equation}
\label{key}
r_{n,p} = \frac{\sum_{k=1}^p \frac{1}{(n-k+1)^2} }{ \sum_{k=1}^p \frac{1}{n-k+1} }  \leq  \max \left\{\frac{1}{\log n+1}
,  \frac{\frac{\pi^2}{6}}{\log(n+1)-\log(\log n+1)} \right\} \rightarrow 0
\end{equation}
uniformly in $p(n)$ as $n \goto \infty$, where $p(n)$ can grow with $n$, $p(n) \leq n$.
\end{lemma}
\begin{proof}
%The proof uses a simple but elegant  fact: If every subsequence of the sequence $r_n$ has a further subsequence that converges to 0 , then $r_n$ converges to 0. \\
%Suppose taking any subsequence of the pair sequence $(n,p)$, consider: (1) if the subsequence has $\liminf\limits_{n \goto \infty} \frac{p}{n} = 0$, then take this lower limit subsequence and for this further subsequence, \eqref{key} is of order $O(\frac{1}{n})$; (2) else if $\liminf\limits_{n\goto \infty} \frac{p}{n} =\delta$ and $0<\delta<1$, then also take this lower limit subsequence, \eqref{key} is of order $O(\frac{1}{n-p})$; (3) else if $\delta = 1$, take the lower limit subsequence and further consider: (3.1) if the $\limsup\limits_{n\goto \infty} (n-p) =\infty$, take this upper limit further subsequence, \eqref{key} is of order $o(\frac{1}{n-p})$  or (3.2) else $\limsup\limits_{n\goto\infty} (n-p) = \eta < \infty$, take this upper limit subsequence, \eqref{key} is of order $O(\frac{1}{\log n})$. In general, no matter what, we can always pick a further subsequence of the subsequence that have \eqref{key} goes to zero. Then the whole sequence converges to zero.
We consider the following two scenarios (1) when $n-p(n) \geq \log n$. (2) when $n-p(n) \leq \log n$.
\\
For case (1), the equation \eqref{key} can be bounded in the following way
\begin{align}
r_{n,p} &= \frac{\sum_{k=1}^p \frac{1}{(n-k+1)^2} }{ \sum_{k=1}^p \frac{1}{n-k+1} } \nonumber \\
&\leq \frac{\frac{1}{n-p+1} \cdot \sum_{k=1}^p \frac{1}{n-k+1} }{ \sum_{k=1}^p \frac{1}{n-k+1} } \leq \frac{1}{\log n+1}.
\end{align}
For case (2), the equation \eqref{key} can be bounded in the following way
\begin{align}
r_{n,p} &= \frac{\sum_{k=1}^p \frac{1}{(n-k+1)^2} }{ \sum_{k=1}^p \frac{1}{n-k+1} }\nonumber \\
& \leq \frac{\frac{\pi^2}{6}}{\sum_{k=1}^p \log(1+\frac{1}{n-k+1}) } \nonumber \\
& \leq \frac{\frac{\pi^2}{6}}{\log(n+1)-\log(n-p+1) } \leq \frac{\frac{\pi^2}{6}}{\log(n+1)-\log(\log n+1) }.
\end{align}
Thus, we have the following bound for $r_{n,p}$ uniformly in $p(n)$
\begin{equation}
r_{n,p} \leq \max \left\{\frac{1}{\log n+1}
,  \frac{\frac{\pi^2}{6}}{\log(n+1)-\log(\log n+1)} \right\} \rightarrow 0, \quad \mbox{as} ~ n \rightarrow \infty.
\end{equation}
Basically we show this sequence converges to 0 uniformly faster than the $O(1/\log n)$ rate.
\end{proof}

\medskip
It follows from Lemmas \ref{StandardLma} and \ref{BartlettLma} that
\begin{equation}
\hat{T} - \log \det \Sigma = \log \det \left( n \hat{I}\right) - \sum\limits_{k=1}^p \left[ \psi\left(\frac{1}{2}(n-k+1)\right)+\log 2\right] \stackrel{\bigtriangleup}{=} Z.
\end{equation}
Thus,
\begin{equation}
Z \stackrel{L}{=} \sum_{k=1}^p \left[\log \left( \chi^2_{n-k+1}\right)-\psi\left(\frac{1}{2}(n-k+1)\right) -\log 2\right].
\end{equation}
Inspired by \cite{jonsson1982some} (where a special case of our theorem has been proved under much stronger conditions) and using the fact of the independence and the characteristic function of the logarithm Chi-square distribution, the characteristic function of $Z$ is
\bea
\phi_Z(t) & = & \prod\limits_{k=1}^p \phi_{\log \chi^2_{n-k+1}}(t) \cdot \frac{1}{{\rm exp}\left(it \cdot [\psi\left(\frac{1}{2}(n-k+1)\right) +\log 2] \right)} \nonumber\\
& = & \prod\limits_{k=1}^p \mathbb{E} e^{it \log \chi^2_{n-k+1} } \cdot \frac{1}{{\rm exp}\left(it \cdot [\psi\left(\frac{1}{2}(n-k+1)\right) +\log 2] \right)}\nonumber\\
& = & \prod\limits_{k=1}^p \mathbb{E} (\chi^2_{n-k+1} )^{it} \cdot  \frac{1}{{\rm exp}\left(it \cdot \psi\left(\frac{1}{2}(n-k+1)\right)  \right)\cdot 2^{it}} \nonumber\\
& = & \prod\limits_{k=1}^p\frac{\Gamma(\frac{1}{2}(n-k+1)+it)}{\Gamma(\frac{1}{2}(n-k+1))} 2^{it} \cdot  \frac{1}{{\rm exp}\left(it \cdot \psi\left(\frac{1}{2}(n-k+1)\right)  \right)\cdot 2^{it}} \nonumber\\
& = & \prod\limits_{k=1}^p\frac{\Gamma(\frac{1}{2}(n-k+1)+it)}{\Gamma(\frac{1}{2}(n-k+1))} \cdot  \frac{1}{{\rm exp}\left(it \cdot \psi\left(\frac{1}{2}(n-k+1)\right)  \right)} .
\eea
Thus we have,
\[
\log \phi_Z(t)  =  \sum\limits_{k=1}^p \left\{ \log \Gamma(\frac{1}{2}(n-k+1)+it) - \log \Gamma(\frac{1}{2}(n-k+1)) - it \cdot \psi\left(\frac{1}{2}(n-k+1)\right) \right\}. \\
\]
Using Taylor expension of Gamma and Digamma function \cite{bateman1981higher}, we have
\beas
\log \Gamma(z) &=& z\log z - z - \frac{1}{2}\log \frac{z}{2\pi} + \frac{1}{12z} + O(\frac{1}{|z|^2}) \\
\psi(z) &=& \log z - \frac{1}{2z} + O(\frac{1}{|z|^2}).
\eeas
Thus for each term in above characteristic function, we have
\bea
 && \log \Gamma(\frac{1}{2}(n-k+1)+it) - \log \Gamma(\frac{1}{2}(n-k+1)) - it \cdot \psi\left(\frac{1}{2}(n-k+1)\right) \nonumber\\
 & =& it \log(\frac{1}{2}(n-k)+1) -it \frac{1}{n-k+1}+ (it)^2 \frac{1}{n-k+1} - it \log(\frac{1}{2}(n-k+1)) \nonumber\\
 && + it \frac{1}{n-k+1} + O(\frac{1}{(n-k+1)^2}) \nonumber\\
 &=& \; (it)^2 \frac{1}{n-k+1} + O(\frac{|t|^2}{(n-k+1)^2}).
\eea
The characteristic function $\phi_0(t)$ of $\frac{1}{\sigma_{n,p}}\left(\log \det \hat{\Sigma} - \tau_{n,p} - \log \det \Sigma \right)$ is
\[
\phi_0(t) = {\rm exp}\left\{\frac{(it)^2}{2} + O(|t|^2 \cdot \frac{\sum_{k=1}^p \frac{1}{(n-k+1)^2} }{ \sum_{k=1}^p \frac{1}{n-k+1} })\right\}
\]
Lemma \ref{real.lemma} shows that
\begin{equation}
r_n = \frac{\sum_{k=1}^p \frac{1}{(n-k+1)^2} }{ \sum_{k=1}^p \frac{1}{n-k+1} } \rightarrow 0
\end{equation}
under $n\rightarrow \infty$ and $p\leq n$. Thus, when $n \rightarrow \infty$, $\phi_0(t) \rightarrow e^{\frac{(it)^2}{2}}$ and the result follows. \qed

\subsection{Proof of Theorem \ref{NAsymUpperThm}}
It follows from the variance of the logarithm Chi-square distribution and the Taylor expansion for TriGamma function that
\[
\psi'(z) = \frac{1}{z} + \left( \frac{1}{2z^2} + \frac{1}{6z^3} \right) \theta
\]
for $z \geq 1$ and $0 < \theta <1$.
Hence,
\begin{align}
\mathbb{E} \left( \hat{T} - \log \det \Sigma \right)^2 & =  {\rm Var}(\sum_{k=1}^p \log(\chi^2_{n-k+1})) \nonumber\\
& =  \sum_{k=1}^p \psi'\left( \frac{n-k+1}{2} \right) \nonumber\\
& =  \sum_{k=1}^p \left[ \frac{2}{n-k+1} + \frac{2 \theta}{(n-k+1)^2} +\frac{4\theta}{3(n-k+1)^3}\right] \nonumber\\
& \leq  \sum_{k=1}^p \left[ -2\log(1-\frac{1}{n-k+1}) + \frac{10}{3(n-k+1)^2}\right] 
\end{align}
Since $\sum_{k=1}^p \frac{1}{(n-k+1)^2} \leq \sum_{k=1}^p \frac{1}{(n-k)(n-k+1)} =\sum_{k=1}^p \left( \frac{1}{n-k} - \frac{1}{n-k+1} \right) = \frac{1}{n-p} - \frac{1}{n} $ and $ \sum_{k=1}^p  \log(1-\frac{1}{n-k+1}) =  \sum_{k=1}^p  \log(\frac{n-k}{n-k+1}) = \log(1-\frac{p}{n})$, we have
\begin{align}
\mathbb{E} \left( \hat{T} - \log \det \Sigma \right)^2 & \leq  -2 \log(1-\frac{p}{n}) + \frac{10}{3} \cdot (\frac{1}{n-p}-\frac{1}{n}) \nonumber\\
& =  -2 \log(1-\frac{p}{n}) + \frac{10p}{3n} \cdot \frac{1}{n-p}. \qed
\end{align}

\subsection{Proof of Theorem \ref{CRLowerThm} }

We first recall the biased version of Cramer-Rao Inequality in multivariate case.
Let  $\Theta_{1 \times p}$ be a parameter vector and let $X \sim f(\Theta)$, where $f(\Theta)$ is the density function.
Consider any estimator $\hat{T}(X)$ of the function $\phi(\Theta)$ with the bias $B(\Theta) = \mathbb{E} \hat{T}(X) - \phi(\Theta) = (b(\theta_1),\ldots,b(\theta_p))^T$. Then
\bea
&&\mathbb{E}_\Theta (\hat{T} - \phi(\Theta))^2  =   Var_\Theta( \hat{T}(X) ) + \| B(\Theta) \|_2^2 \nonumber\\
&& \geq \; \left( \frac{\partial  \left( \phi(\Theta) + B(\Theta) \right)}{\partial \Theta} \right)^T \cdot \left[ \mathbf{I}(\Theta) \right]^{-1} \cdot \frac{\partial \left( \phi(\Theta) + B(\Theta) \right)}{\partial \Theta}+ B(\Theta)^T  \cdot B(\Theta). \quad\quad \label{FisherInfo}
\eea
 Now consider the diagonal matrix subfamily, $\Sigma ={\rm diag}\left(\theta_1, \theta_2, ..., \theta_p\right) $, with $p$ parameters, $\theta_1, ..., \theta_p$.
We wish to estimate $\phi(\Theta) = \sum_{i=1}^p \log(\theta_i) = \log \det(\Sigma)$.
%For $x_i \sim N(0,\theta), ~ i = 1,2,...,n$, for one sample, the fisher information with respect to parameter $\theta$ is
%\[
%\mathbf{I} (\theta)  =  \mathbb{E}_{f(x,\theta)} \left[ - \frac{\partial^2 \log f(x,\theta)}{\partial \theta^2}  \right] =   \mathbb{E}_{f(x,\theta)} \left[ \frac{x^2}{\theta^3} - \frac{1}{2\theta^2} \right] = \frac{1}{2\theta^2}.
%\]
For a random sample $X_1, ..., X_n\stackrel{iid}{\sim} {\mathcal N}_p(0, \Sigma)$, the Fisher information matrix and the partial derivative of the $\phi(\Theta)$ are given by
\[
\mathbf{I}(\Theta) = {\rm diag}\left(\frac{n}{2\theta_1^2},\frac{n}{2\theta_2^2}, \ldots, \frac{n}{2\theta_p^2}\right) \quad\mbox{and}\quad
\frac{\partial \phi(\Theta)}{\partial \Theta} = \left(\frac{1}{\theta_1}, \frac{1}{\theta_2}, \ldots, \frac{1}{\theta_p}\right)^T.
\]
Equation (\ref{FisherInfo}) can be calculated explicitly as
\bea
& & \left( \frac{\partial  \left( \phi(\Theta) + B(\Theta) \right)}{\partial \Theta} \right)^T \cdot \left[ \mathbf{I}(\Theta) \right]^{-1} \cdot \frac{\partial \left( \phi(\Theta) + B(\Theta) \right)}{\partial \Theta}+ B(\Theta)^T  \cdot B(\Theta) \nonumber \\
&&=\;  \left(\frac{1}{\theta_1}+b'(\theta_1), ..., \frac{1}{\theta_p}+b'(\theta_p)\right) \left[ {\rm diag}\left(\frac{n}{2\theta_1^2}, \ldots, \frac{n}{2\theta_p^2}\right) \right]^{-1} \left(\frac{1}{\theta_1}+b'(\theta_1), \ldots, \frac{1}{\theta_p}+b'(\theta_p)\right) ^T \nonumber\\
&& \quad+ \sum_{k=1}^p b^2(\theta_k) \nonumber\\
&&=\;  \sum_{k=1}^p \left[ \frac{2}{n} \left( 1+ \theta_k b'(\theta_k)\right)^2+b^2(\theta_k) \right].\label{LowerBound}
\eea
As in \cite{brown1991information}, if we can prove that for any bias function $b(\theta)$
\begin{equation}
\sup_{\theta>0}  \left[ \frac{2}{n} \left( 1+ \theta b'(\theta)\right)^2+b^2(\theta) \right] \geq \frac{2}{n}, \label{InfoLowerBound}
\end{equation}
then the minimax lower bound result 
\begin{equation}
\inf\limits_{\hat{T}} \sup_{\Sigma\in \sf} \mathbb{E} (\hat{T} - \log \det \Sigma)^2 \geq 2 \cdot \frac{p}{n}
\end{equation}
 holds  for any  parameter space $\mathcal{F}$ containing the set of the diagonal matrices by combining (\ref{LowerBound}) and  (\ref{InfoLowerBound}).

To prove equation (\ref{InfoLowerBound}), we first prove that for any given constant $K>0$
\begin{equation}
 \sup_{1/K \leq \theta\leq K}  \left[ \frac{2}{n} \left( 1+ \theta b'(\theta)\right)^2+b^2(\theta) \right] \geq \frac{2}{n} \left( \frac{\log K}{\log K+\sqrt{\frac{2}{n}}} \right)^2. \label{LowerIneq}
\end{equation}
Assume 
\[
r_K \geq  \sup\limits_{1/K \leq \theta\leq K}  \left[ \frac{2}{n} \left( 1+ \theta b'(\theta)\right)^2+b^2(\theta) \right],
\]
 then we have the following two inequalities
\[
|1+\theta b'(\theta)| \leq \sqrt{\frac{n}{2}\cdot r_K}\quad\mbox{ and }\quad
|b(\theta)| \leq \sqrt{r_K},
\]
which implies $r_K \geq \frac{2}{n} \left( \frac{\log K}{\log K+\sqrt{\frac{2}{n}}} \right)^2 $. This means that $\frac{2}{n} \left( \frac{\log K}{\log K+\sqrt{\frac{2}{n}}} \right)^2$ is a lower bound for (\ref{LowerIneq}). Equation \eqref{InfoLowerBound} now follows by letting $K \goto \infty$. \qed

\subsection{Proof of Theorem \ref{AsympOptThm}}
The upper bound given in Theorem \ref{NAsymUpperThm} yields that
\[
 \inf_{\delta} \sup_{\Sigma} \mathbb{E} (\delta - \log \det \Sigma)^2 \leq -2 \cdot \log\left( 1- \frac{p}{n}\right) + \frac{p}{n} \cdot \frac{10}{3(n-p)}.
\]
It then follows from the assumption $n-p\rightarrow \infty$ that
\[
 \varlimsup_{n \rightarrow \infty} \inf_{\hat{T}} \sup_{\Sigma} \mathbb{E} (\hat{T} - \log \det \Sigma)^2 \leq 2 \cdot \limsup_{n \rightarrow \infty} -\log\left( 1- \frac{p}{n}\right).
\]
When $r=0$, $-2 \cdot \log\left(1- \frac{p}{n}\right) \sim 2\cdot \frac{p}{n}$ and the upper bound follows.
For the lower bound, Theorem \ref{CRLowerThm} implies
\[
 \varliminf_{n \rightarrow \infty} \inf_{\hat{T}} \sup_{\Sigma} \mathbb{E} (\hat{T} - \log \det \Sigma)^2 \geq 2\cdot \liminf_{n \rightarrow \infty} \frac{p}{n}.
\]
This completes the proof. \qed

\subsection{Proof of Theorem \ref{MinimaxLowerThm}}

The proof uses a two point hypothesis testing argument due to Le Cam (see \cite{tsybakov2008introduction} page 79-80).
\begin{lemma}[Le Cam's Lemma]
\label{LeCam}
Let $\hat{\theta}$ be any estimator of $\theta$ based on an observation from a distribution in the collection $\{\mathbb{P}_{\theta_0}, \mathbb{P}_{\theta_1}\}$, suppose $|\theta_0 - \theta_1| \geq 2s $, then
\begin{equation}
\inf_{\hat{\theta}} \sup_{\theta \in \{\theta_0,\theta_1\}} \mathbb{E} (\hat{\theta}_n - \theta)^2 \geq s^2 \cdot \frac{1}{2} \|\mathbb{P}_{\theta_0} \wedge \mathbb{P}_{\theta_1} \| 
\end{equation}
where $\|\mathbb{P} \wedge \mathbb{Q}\| = \int (p \wedge q) d\mu$, is affinity between probability measures.
\end{lemma}
The total variance affinity can be lower bounded in terms of the  $\chi^2$ distance.
\begin{lemma}[Pinsker's Inequality]
\label{Pinsker}
\begin{equation}
\|\mathbb{P} \wedge \mathbb{Q}\| = 1 - TV(\mathbb{P}, \mathbb{Q}) \geq 1 - \sqrt{KL(\mathbb{P}, \mathbb{Q})/2}  \geq 1 - \sqrt{\chi^2(\mathbb{P}, \mathbb{Q})/2}
\end{equation}
where $TV(\mathbb{P}, \mathbb{Q}) = \frac{1}{2}\int |p-q| d\mu$ is the total variation distance, $KL(\mathbb{P}, \mathbb{Q})$ is the Kullback-Leiber divergence, $\chi^2(\mathbb{P}, \mathbb{Q})$ is the $\chi^2$ distance.
\end{lemma}
We use the follow lemma,  which is a direct consequence of Lemma 2 in \cite{cai2011minimax}, to bound the $\chi^2$ distance.
\begin{lemma}
For $i=0$ and $1$, let $\mathbb{P}_i$ be the joint distribution of $n$ independent $p$-dimensional Gaussian variables with the covariance matrix $\Sigma_i$. The  $\chi^2$ distance $\chi^2(\mathbb{P}_0,\mathbb{P}_1)$ satisfies
\begin{equation}
\chi^2(\mathbb{P}_0,\mathbb{P}_1)+1 = \int \frac{P_1^2}{P_0} d\mu = \left\{ \det\left( I - (\Sigma_1-\Sigma_0)\Sigma_0^{-1}(\Sigma_1-\Sigma_0)\Sigma_0^{-1} \right)\right\}^{-n/2}
\end{equation}
\end{lemma}
To prove the lower bound given in \ref{MinimaxLowerThm},
we pick $\Sigma_0= I_{p\times p}$, $\Sigma_1 = (1+\frac{1}{\sqrt{np}}) \cdot I_{p\times p}$. \\
Firstly, let's prove the theorem under $np>\max\{\frac{1}{(K-1)^2},1\}$. It is easy to see that this two points lie in the parameter space because $1/K < 1+\frac{1}{\sqrt{np}} < K$.
Then
\begin{equation}
|\theta_0 -\theta_1|= |\log \det \Sigma_0 - \log\det \Sigma_1| = p \log\left(1+\sqrt{\frac{1}{pn}}\right) > p \frac{\sqrt{\frac{1}{pn}}}{1+\sqrt{\frac{1}{pn}}} \geq \frac{1}{2} \sqrt{\frac{p}{n}}
\end{equation}
\begin{equation}
\begin{array}{rcl}
\chi^2(\mathbb{P}_0,\mathbb{P}_1)+1 &=& \left\{ \det\left( I - (\Sigma_1-\Sigma_0)\Sigma_0^{-1}(\Sigma_1-\Sigma_0)\Sigma_0^{-1} \right)\right\}^{-n/2}\\
& = & \left(1-\frac{1}{np}\right)^{-\frac{1}{2}np} < e^{\frac{1}{2}np \frac{\frac{1}{np}}{1-\frac{1}{np}}} < e <\infty \quad \mbox{For $np>1$}.
\end{array}
\end{equation}
The $\chi^2$ distance is upper bounded away from infinity, thus the affinity term is lower bounded away from 0. At the same time, the parameters are well separated away with a distrance $s = \frac{1}{4}\sqrt{\frac{p}{n}}$.
Thus, by Le Cam's Lemma, we have, for some constant $c>0$ ($c\leq 2$ is due to the Theorem \ref{CRLowerThm})
\begin{equation}
\inf_{\delta}\sup_{\Sigma \in \mathcal{D}_K} \mathbb{E} \left( \delta - \log \det \Sigma \right)^2 \geq  \left(\frac{1}{4}\sqrt{\frac{p}{n}}\right)^2 \cdot \frac{1}{2} \left(1 - \sqrt{\frac{e-1}{2}}\right)  = c \cdot \frac{p}{n}  ,
\end{equation}
for all $p,n$ as long as $np>\max\{\frac{1}{(K-1)^2},1\}$. More specifically, $c$ can be taken as $\frac{1}{32} \left(1 - \sqrt{\frac{e-1}{2}}\right)$. \\
Secondly, for $np \leq \max\{\frac{1}{(K-1)^2},1\}$, there are only finite collection of $(n,p)$ pairs, thus we must have a constant $c_K$ small enough such that
\begin{equation}
\inf_{\delta}\sup_{\Sigma \in \mathcal{D}_K} \mathbb{E} \left( \delta - \log \det \Sigma \right)^2 \geq  c_K \cdot \frac{p}{n}.
\end{equation}
Thus combining two parts, we can pick $C_K = \min\{c_K, c \}$, 
 which completes the proof. \qed

\bigskip

%\subsection{Proof of Theorem \ref{MinimaxOptimalThm}}

%It is easy to see if we combine of Theorem (\ref{MinimaxLowerThm}) and Theorem (\ref{NAsymUpperThm}). Under $\lim\limits_{n\rightarrow \infty} \frac{p}{n} = r, ~ 0 \leq r <1$, we have the upper bounds is dominated by $-2 \log\left( 1- \frac{p}{n} \right) = O (\frac{p}{n})$. \qed

%\bibliographystyle{apalike}
%\bibliographystyle{plainnat}
\bibliographystyle{ieeetr}
\bibliography{Covariance-Determinant-081513}

\end{document}